\documentclass[12pt, 14paper]{amsart}
\usepackage{amsmath,amsfonts,amssymb}

\usepackage{longtable}
\usepackage[mathscr]{eucal}
\usepackage{amsmath, amsthm}
\usepackage{mathrsfs}
\usepackage{amsbsy}
\usepackage{wasysym}
\usepackage{url}

\theoremstyle{plain}
\usepackage{color}

\newtheorem{theorem}[subsection]{Theorem}

\newtheorem{lemma}[subsection]{Lemma}
\newtheorem{sublemma}[subsection]{Sublemma}

\newtheorem{corollary}[subsection]{Corollary}

\input xypic
\xyoption{all}
\usepackage[breaklinks]{hyperref}
\theoremstyle{thmrm}

\newcommand \ZZ {{\mathbb Z}}

\newcommand \PR {{\mathbb P}}

\newcommand \QQ {{\mathbb Q}}

\newcommand \bcL {{\mathscr L}}

\newcommand \bcU {{\mathscr U}}
\newcommand \bcV {{\mathscr V}}

\newcommand \Pic {{\rm {Pic}}}

\newcommand \Sym {{\rm {Sym}}}

\newcommand \pr {{\rm {pr}}}
\newcommand \ev {{\it {ev}}}

\newcommand \im {{\rm im}}

\newcommand \Hom {{\rm Hom}}

\newcommand \id {{\rm {id}}}

\newcommand \CH {{\it {CH}}}
\newcommand \wt {\widetilde }

\newcommand \Th {\Theta}

\usepackage{graphicx}
\begin{document}

\title[Zero cycles onn Prym varieties]{Zero cycles on Prym varieties}
\author{ Kalyan Banerjee}

\address{VIT Chennai, India}

\email{kalyan.banerjee@vit.ac.in
}

\footnotetext{Mathematics Classification Number: 14C25, 14D05, 14D20,
 14D21}
\begin{abstract}
In this text we prove that if an abelian variety $A$ admits  an embedding into the Jacobian of a smooth projective curve $C$, and if we consider $\Th_A$ to be the divisor $\Th_C\cap A$, where $\Th_C$ denotes the theta divisor of $J(C)$, then the embedding of $\Th_A$ into $A$ induces an injective push-forward homomorphism (under certain conditions) at the level of Chow groups. We show that this is the case for every Prym varietiy arising from an unramified double cover of smooth projective curves. As a consequence that there does not exist a universal codimension two cycle on the product of a very general cubic threefold and the Prym variety associated to it. Hence we conclude that a very general cubic threefold is stably irrational.
\end{abstract}

\maketitle

\section{Introduction}
In the paper \cite{BI} the authors were investigating the following question. Let $C$ be a smooth projective curve of genus $g$ and let $\Th$ denote the theta divisor embedded into the Jacobian $J(C)$ of the curve $C$. Let $j$ denote this embedding. Then the push-forward homomorphism $j_*$ at the level of Chow groups of zero cycles is injective.

In this paper we investigate the following question for an arbitrary principally polarized abelian variety $A$. Let $A$ be a principally polarized abelian variety and let $H$ denote a divisor embedded inside $A$. Let $j$ denote this embedding. Then what can we say about the kernel of the push-forward homomorphism $j_*$ at the level of Chow groups of $0$-dimensional cycles? This question is affirmatively answered in certain cases  when the abelian variety $A$ is embedded in the Jacobian variety $J(C)$ of some smooth projective curve $C$  and we consider the divisor $\Th\cap A$ inside $A$, where $\Th$ is the theta divisor of $J(C)$, with the assumption of some extra conditions on the inverse image of the abelian variety under the natural map $\Sym^g C\to J(C)$. It is well-known that any principally polarized abelian variety is embedded in the Jacobian of a smooth projective curve (see \cite{BL} corollary $12.2.4$).So for the principally polarized abelian varieties with the condition on the inverse image of the abelian variety mentioned above, the above question about the kernel of the push-forward homomorphism at the level of Chow groups of $0$-dimensional cycle is answered, when the divisor $H$ is the intersection of the abelian variety $A$ with the theta divisor of the ambient Jacobian variety, where $A$ is embedded.

\begin{theorem}(= Theorem \ref{theorem4})Let $A$ be a principally polarized abelian variety embedded into the Jacobian $J(C)$ of a smooth projective curve $C$. Suppose that $A$ is a proper subvariety of $J(C)$. Let $\Th_C$ denote the theta divisor of $J(C)$. Pull it back to $A$ and denote the pull-back by $\Th_A$. Suppose that $\Sym^{g-i}C\cap A'$ is singular  for some $i\geq 0$ and suppose also that $C\cap A'$ is  irreducible, where $A'$ is the inverse image of $A$, under the map $\Sym^g C\to J(C)$. Then for the closed embedding $j:\Th_A\to A$,  $j_*$ has torsion kernel at the level of Chow groups of zero cycles.
\end{theorem}

So in answering our question for an arbitrary principally polarized abelian variety $A$ and the divisor $\Th_A=\Th_C\cap A$, it is enough to assume that $C'\cap A$ is irreducible.

As an application we get that the embedding of the theta divisor inside a Prym variety induces injection at the level of Chow groups of zero cycles with $\QQ$-coefficients. We also show that if we start with a principally polarized abelian surface $A$ embedded in $J(C)$ (with the condition that $A'\cap C$ is irreducible) and consider the corresponding $K3$-surface, then the push-forward at the level of Chow groups of zero cycles (with $\QQ$-coefficients) induced by the closed embedding of the divisor associated to $\Th_C\cap A$ inside the $K3$-surface is injective.

Another application of the main theorem is the non-existence of a universal codimension two cycle on the product of a very general smooth, cubic threefold and the Prym variety associated to it. As a consequence the integral Chow theoretic decomposition of the diagonal does not hold on the two fold product of a very smooth, general cubic threefold and hence a very general smooth, cubic threefold is stably irrational.

{\small \textbf{Acknowledgements:} The author is indebted to A.Beauville for communicating the main application of the paper. The author expresses his sincere gratitude to Jaya Iyer for telling this problem about injectivity of push-forward induced by the closed embedding of a divisor into an abelian variety, to the author and for discussions relevant to the theme of the paper. The author thanks C.Voisin for pointing out the fact that this injectivity of the push-forward is not true for  zero cycles on a very general divisor on an abelian variety and also for relevant discussion regarding the theme of the paper. Author is grateful to the anonymous referee for mentioning the non-injectivity phenomena in the main theorem of the paper, when we consider the embedding of an elliptic curve into a Jacobian and also for several other improvements. Finally the author wishes to thank Olivier de Gaay Fortman concerning a crucial point in the final argument about the singularity of the theta divisor. The author wishes to thank the ISF-UGC and DAE grant for funding this project and hospitality of Indian Statistical Institute, Bangalore Center, HRI Allahabad, VIT Chennai for hosting this project.}

\section{Abelian varieties embedded in Jacobians}
We work over an algebraically closed ground field $k$ of characteristic zero.

First let  us fix some notation. Let $P_0$ be a $k$-rational point on the smooth projective curve $C$. Then define the closed embedding of $\Sym^m C$ into $\Sym^n C$ for $m\leq n$ by the rule:

$$[x_1,\cdots,x_m]\mapsto [x_1,\cdots,x_m,P_0,\cdots,P_0]$$
here $[x_1,\cdots,x_m]$ denote the unordered $m$-tuple of $k$-points on $C$.
Let us defined the map from $\Sym^n C$ to $J(C)$ by the rule:
$$[x_1,\cdots,x_n]\mapsto \sum_i x_i-nP_0\;.$$

Let $A$ be an abelian variety embedded inside the Jacobian of a  smooth projective curve $C$. Let $\Th_C$ denote the theta divisor of the Jacobian $J(C)$. Consider $\Th_A$ to be $\Th_C\cap A$, and the closed embedding of $\Th_A$ into $A$, denote it by $j_A$, then we prove that $j_{A*}$ from $\CH_0(\Th_A)$ to $\CH_0(A)$ has torsion kernel.

To prove that, first we show that the embedding of $\Th_C$ into $J(C)$ gives rise to an injection at the level of Chow groups of zero cycles. Although this has been proved in \cite{BI}[theorem 3.1], here we present an alternative proof following \cite{Collino} which gives a better understanding of the picture when we blow up $J(C)$ along some subvariety (in our case finitely many points).

\begin{theorem}
\label{theorem1}
The embedding of the theta divisor $\Th_C$ into $J(C)$ for a smooth projective curve $C$, gives rise to an injection at the level of Chow groups of zero cycles.
\end{theorem}
\begin{proof}
We prove this by induction. Let $\sum_{i=1}^m C$ be the image of $\Sym^m C$, under the natural map from $\Sym^m C$ to $J(C)$. We prove : for any $m\leq g-1$ the closed embedding of $\sum_{i=1}^m C$ into $J(C)$ gives rise to an injective homomorphism at the level of Chow group of zero cycles.
We use the fact that $\Sym^g C$ maps surjectively and birationally onto $J(C)$ and $\Sym^{m}C$ maps surjectively and birationally onto $\sum_{i=1}^m C$. Let $\pr$ be the natural projection from $C^g$ to $C^m$ given by
$$(x_1,\cdots,x_g)\mapsto (x_1,\cdots,x_m)\;.$$

We have a natural correspondence $\Gamma$ given by $\pi_g\times \pi_{m}(Graph(\pr))$, where $\pi_i$ is the natural quotient morphism from $C^i$ to $\Sym^i C$. Consider the correspondence $\Gamma_1$ on $J(C)\times \sum_{i=1}^m C$ given by
$(f_2\times f_1)(\Gamma)$,
where $f_1,f_2$ are natural morphisms from $\Sym^{m}C,\Sym^g C$ to $\sum_{i=1}^m C,J(C)$ respectively. Now we define the homomorphism $\Gamma_{1*}$ from $\CH_0(J(C))$ to $\CH_0(\sum_{i=1}^m C)$ as follows:
$$\Gamma_{1*}(z)=\pr_{2*}(z\times \sum_{i=1}^m C.\Gamma_1)$$
where
$$(z\times \sum_{i=1}^m C).\Gamma_1:=(f_2\times f_1)_*[(f_2\times f_1)^*(z\times \sum_{i=1}^m C).(f_2\times f_1)^*\Gamma_1]$$
this is to make $\Gamma_{1*}$ well defined, as $J(C)\times \sum_{i=1}^m C$ is singular. So we have to pull the cycles back to $\Sym^g C\times \Sym^{m}C$ to have a well defined intersection product. Since $f_2,f_1$ are not flat we need to define what we mean my $(f_2\times f_1)^*(z\times \sum_{i=1}^m C)$, for any zero cycle $z$. So for $z$ supported in the complement of the exceptional locus of the birational map $f_2$, we have
$$(f_2\times f_1)^*(z)=z'\times \Sym^m C$$
here $z'$ is $f_2^{*}(z)$, this is a zero cycle as $f_2^{-1}$ restricted to the complement of the  exceptional locus is an isomorphism, hence it takes zero cycles to zero cycles.
Let $V_1$ be the Zariski open subset of $J(C)\times \sum_{i=1}^m C$ mapping isomorphically onto the Zariski open set $V_1'$ of $\Sym^g C\times \Sym^m C$.

\begin{lemma}
\label{lemma4}
The pullback cycle $(f_2\times f_1)^*(\Gamma_1)$ is well defined and it is rationally equivalent to $\Gamma+S$, where $S$ is supported on the complement $V_1'$ and it is of same dimension as $(f_2\times f_1)^*(\Gamma_1)\;.$
\end{lemma}
We define
$$(f_2\times f_1)^*(\Gamma_1):=[\overline {((f_2\times f_1)|_{V_1}^{*}({\Gamma_1}|_{V_1}))}]$$
Let us explain this. Since $V_1'\to V_1$ is an isomorphism, the pullback $(f_2\times f_1)|_{V_1}^*$ is well defined from $\CH^m(V_1)$ to $\CH^m(V_1')$. Therefore we restrict the correspondence $\Gamma_1$ to $V_1$ and consider its pull-back under ${(f_2\times f_1)}^*|_{V_1}$ from $\CH^m(V_1)$ to $\CH^m(V_1')$. Then consider the Zariski closure of each of the components of $(f_2\times f_1)^*|_{V_1}(\Gamma_1|_{V_1})$ in $\Sym^g C\times \Sym^m C$. Consider the rational equivalence class of the cycle defined by taking Zariski closures of each of the components with multiplicity, that is consider:
$$[\overline{ ((f_2\times f_1)|_{V_1}^{*}({\Gamma_1}|_{V_1}))}]\;.$$

Now consider the following commutative diagram, where the rows are exact:

$$
  \xymatrix{
     \CH_{g}(E') \ar[r]^-{j'_{*}} \ar[dd]_-{}
  &   \CH_g(\Sym^g C\times \Sym^m C) \ar[r]^-{} \ar[dd]_-{}
  & \CH_g(V_1')  \ar[dd]_-{}  \
  \\ \\
   \CH_g(E'') \ar[r]^-{}
    & \CH_g(J(C)\times \sum_{i=1}^m C) \ar[r]^-{}
  & \CH_0(V_1)
  }
$$
Here the rightmost arrows in the rows are onto, $E',E''$ are the complement of $V_1',V_1$ in $\Sym^g C\times \Sym^m C$ and $J(C)\times \sum_{i=1}^m C$ respectively. The right vertical map is an isomorphism. Therefore if we consider the cycle
$$(f_2\times f_1)^*(\Gamma_1)-\Gamma$$
restricted on $V_1'$, then it is zero because:
$${(f_2\times f_1)^*(f_2\times f_1)}_*(\Gamma)|_{V_1}$$
is equal to
$$(f_2\times f_1)^*|_{V_1}(f_2\times f_1)|_{V'_1*}(\Gamma|_{V_1'})=\Gamma|_{V_1'}\;.$$
Therefore by the localization exact sequence of  the first row of the above diagram we have that there exists $S$ supported on $\CH_g(E')$, such that
$$(f_2\times f_1)^*(\Gamma_1)=\Gamma+S\;.$$

 So we have
$$(f_2\times f_1)^*(z\times \sum_{i=1}^m C).(f_2\times f_1)^*\Gamma_1=z'\times \Sym^m C.(\Gamma+S)$$
where $z'\times \Sym^m C.S$ is zero because $z'.E=0$, for $E$ being the exceptional locus of the map $f_2$. Therefore the above intersection is
$$z'\times \Sym^m C.\Gamma\;.$$
Now we have to check the following:

\begin{lemma}
\label{lemma2}
The above intersection product is well defined.
\end{lemma}

\begin{proof}
So suppose that we have $z_1$ rationally equivalent to $z_2$ on $J(C)$ and supported on the Zariski open set $V$, which maps isomorphically to the Zariski open set $V'$ of $\Sym^g C$. Then we have to prove that the lifts $z_1',z_2'$ of $z_1,z_2$ supported on $V'$ are rationally equivalent on $\Sym^g C$. So let $C_1,\cdots,C_n$ be irreducible normal curves such that their union contains the support of $z_1-z_2$ and there exist rational functions $f_i$ on $C_i$ such that
$$z_1-z_2=\sum_{i=1}^n j_{i*}[div(f_i)]$$
where $[div(f_i)]$ is the class of the divisor associated to the rational function $f_i$ on $C_i$ and $j_i$ is the inclusion of $C_i$ into $J(C)$. Now here we have to prove the following:

\begin{sublemma}
\label{sublemma1}
Given $z_1-z_2$ as above, we always have a finite collection of smooth projective curves $C_i$ and rational functions $f_i$ on $C_i$, such that each $C_i$ intersect the exceptional locus $E$ of the above bi-rational map at finitely many points and we have
$$z_1-z_2=\sum_{i=1}^n j_{i*}[div(f_i)]\;.$$
\end{sublemma}

\begin{proof}
Consider the collection of curves and rational functions on them say $\{(C_i,f_i)\}_i$ such that the above relation
$$z_1-z_2=\sum_i j_{i*}[div(f_i)]$$
holds.
By \cite{Fulton}[example 1.6.3] the above is equivalent to the following:
there exists a morphism from $\PR^1_k$ to $\Sym^n J(C)$ for some $n$ and a positive zero cycle $z$ on $J(C)$ such that
$$f(0)=z_1+z, \quad f(\infty)=z_2+z\;.$$
That is given a collection of $\{(C_i,f_i)\}_i$ such that
$$z_1-z_2=\sum_i j_{i*}[div(f_i)]$$
there exists a morphism $f$ from $\PR^1_k$ to a symmetric power $\Sym^n J(C)$ and $z$ a positive zero cycle such that the above is satisfied and vice versa.
Now consider the collection of all $(f,z)$ such that $f\in \Hom^v(\PR^1_k,\Sym^{n+d} J(C))$, where the superscript $v$ stands for the degree of the morphism from $\PR^1_k$ to $\Sym^{n+d}J(C)$ and $z$ in $\Sym^d J(C)$ such that
$$f(0)=z_1+z,\quad f(\infty)=z_2+z\;.$$ We prove that this collection is a countable union of Zariski closed subsets in the product $\Hom^v(\PR^1_k,\Sym^{n+d}J(C))\times \Sym^d J(C)$. It is enough to consider $z_1,z_2$ to be effective cycles of the same degree $n$. This is because the effective cycles $z_1+z,z_2+z$ are connected by a rational curve on the same symmetric power $\Sym^{n+d}J(C)$. Therefore
$$\deg(z_1+z)=\deg(z_2+z)$$
which gives
$$\deg(z_1)+\deg(z)=\deg(z_2)+\deg(z)$$
implying that
$$\deg(z_1)=\deg(z_2)\;.$$

Consider the following morphisms:
$$\phi:\Sym^n J(C)\times\Sym^n J(C)\times \Sym^d J(C)\to \Sym^{n+d}J(C)\times \Sym^{n+d}J(C)$$
given by
$$(z_1,z_2,z)\mapsto (z_1+z,z_2+z)$$
and
$$\ev:\Hom^v(\PR^1_k,\Sym^{n+d}J(C))\to \Sym^{n+d}J(C)\times \Sym^{n+d}J(C)$$
given by
$$f\mapsto (f(0),f(\infty))$$
Then consider the fiber product of $\Hom^v(\PR^1_k,\Sym^{n+d}J(C))$ and $\Sym^n J(C)\times \Sym^n J(C)\times \Sym^d J(C)$ over $\Sym^{n+d}J(C)\times \Sym^{n+d}J(C)$ given by the above two maps. That is we consider all
$$(f,z_1,z_2,z)$$
{such that}
$$f(0)=z_1+z, f(\infty)=z_2+z\;.$$
Call this fiber product  $Z^n_{d,v}$, this admits a projection morphism to $\Sym^n J(C)\times \Sym^n J(C)$. Then for fixed $(z_1,z_2)$, the fiber over it in $Z^n_{d,v}$ is the collection of all $(f,z)$ such that
$$f(0)=z_1+z,\quad f(\infty)=z_2+z\;.$$
Now consider the scheme
$$\bcV^n_{d,v}:=\{(x,f,z)|x\in \im(f), f(0)=z_1+z,f({\infty})=z_2+z\}$$
this is a $\PR^1$-bundle over ${Z^n_{d,v}}_{z_1,z_2}$ which is the fiber of the projection of $Z^n_{d,v}\to \Sym^n J(C)\times \Sym^n J(C)$. Consider the subscheme of $\bcV^n_{d,v}$ given by $\bcV^n_{d,v,E}$ to be the collection of $(x,f)$ in $\bcV^n_{d,v}$ such that image of $f$ intersects $\Sym^{n+d}E$ at finite number of points. Suppose that the projection map from $\bcV^n_{d,v,E}$ to ${Z^n_{d,v}}_{z_1,z_2}$ is surjective (observe that this map has finite fibers) and the dimension of $\bcV^n_{d,v}$ is  $\dim({Z^n_{d,v}}_{z_1,z_2})+1$, which is $\dim(\bcV^n{d,v,E})+1$. So we have
$$\dim(\bcV^n_{d,v})-1=\dim(\bcV^n_{d,v,E})\;.$$
Consider the projection $\pi_{n+d}:\bcV^n_{d,v}\to \Sym^{n+d} J(C)$. Now $\bcV^n_{d,v,E}$ is nothing but the inverse image
$$\pi_{n+d}^{-1}(\Sym^{n+d}E)\;.$$
The Zariski closed set, $\Sym^{n+d}E$ is of codimension bigger than two, say $i$ in $\Sym^{n+d}J(C)$. Let $g$ be the genus of the curve $C$. Observe that ${\Sym^{n+d}}\Sym^g C$ maps surjectively and birationally onto $\Sym^{n+d} J(C)$. So the inverse image of $\Sym^{n+d}E$ under this map is generically a product of projective bundles over $\Sym^{n+d}E$ (because the fibers are product of linear systems of positive dimension). Consider the pullback of this prodcut of projective bundles to $\bcV^n_{d,v,E}$. Denote it by $\widetilde{\bcV^n_{d,v,E}}$. Then the dimension of this pullback $\widetilde{\bcV^n_{d,v,E}}$ is strictly greater than that of $\bcV^n_{d,v,E}$. On  the other hand since ${\Sym^{n+d}}\Sym^g C\to \Sym^{n+d} J(C)$ is birational, the pull-back of $\bcV^n_{d,v}$ along the above map, say $\widetilde{\bcV^n_{d,v}}$ is birational to $\bcV^n_{d,v}$ and it contains the pullback $\widetilde{\bcV^n_{d,v,E}}$.  Also we have
$$\dim(\widetilde{\bcV^n_{d,v}})-1=\dim(\bcV^n_{d,v})-1=\dim(\bcV^n_{d,v,E})\;.$$
This tells us that $\widetilde{\bcV^n_{d,v}}$ is a $\PR^1$-bundle over $\bcV^n_{d,v,E}$, contradicting the fact that $\bcV^n_{d,v}$ is birational to $\widetilde{\bcV^n_{d,v}}$.
Therefore there exists $f$ in $\Hom^v(\PR^1_k,\Sym^{n+d}J(C))$ such that image of $f$ has empty intersection with $\Sym^{n+d}E$. This finishes the proof saying that there exist curves $C_i$ such that $C_i$ is not contained inside $E$ and we have
$$z_1-z_2=\sum_{i=1}^n j_{i*}[div(f_i)]\;.$$
Indeed suppose that there exist curves $C_i$ such that some of them are contained in $E$. Then by the equivalent definition of rational equivalence\cite {Fulton}[chapter 1 on rational equivalence] it follows that the curves $C_i$, gives rise to algebraic cycles (of dimension $1$) $V_i$ on $J(C)\times \PR^1$ mapping dominantly onto $\PR^1$, such that
$$\sum_i V_i[0]=z_1+z\;,$$
$$\sum_i V_i[\infty]=z_2+z\;,$$
such that $z$ is a positive zero cycle.

Here
$$V_i[P]=\pr_{J(C)*}({\pr_{\PR^1}}|_{V_i}^{-1}(P))\;.$$
We have to define $V_i[P]$ this way because the Chow moving lemma may not be true, when we replace $J(C)$ by a singular variety. Now if some of the $C_i$ is contained inside $E$, then we have $V_i$ contained in $E\times \PR^1$. This is because, $V_i$ is obtained as the closure of the graph of the rational function $f_i$ on $C_i$ added to some other specific cycle,\cite{Fulton}[chapter 1]. So if $C_i$ is contained in $E$, then the closure of the graph of corresponding $f_i$ is contained in $E\times \PR^1$. Since $V_i$ maps dominantly onto $\PR^1$ the fibers of these maps are finite(possibly after normalizing $V_i$). Therefore the scheme
$${{\pr_\PR^1}}|_{V_i}^{-1}(P)$$
is a finite set of points. Further if $V_i$ is in $E\times \PR^1$, this finite set of points land inside $E$. Following \cite{Fulton}[chapter 1], this $V_i$'s gives rise to a set theoretic map $f$ from $\PR^1$ To $\Sym^d J(C)$ for some $d$, which is a regular morphism  of schemes. The map is obtained as above:
$$V_i[P]=\pr_{J(C)*}({\pr_{\PR^1}}|_{V_i}^{-1}(P))\;.$$
Since for $C_i$ in $E$, we have, $V_i$ in $E\times \PR^1$, we obtaine that image of $f$ actually intersects $\Sym^d E$.
\end{proof}

Then according to the above theorem there exists a finite collection of smooth projective curves $C_i$ not contained in $E$ and rational functions $f_i$ on them such that
$$z_1-z_2=\sum_i j_{i*}[div(f_i)]\;.$$

 Let $\wt{C_i}$ be the strict transform of $C_i$, under the birational map $f_2$. This means that
$$\wt{C_i}=\overline{f_2^{-1}(C_i\setminus E)}$$
where $E$ is the exceptional locus of the birational map $f_2$. Since $C_i$ is birational to $\wt{C_i}$ and $C_i$ is normal, we have $C_i\cong \wt{C_i}$. So let us consider the map
$$\phi:\cup_i \wt{C_i}\to \cup_i C_i$$
this is an isomorphism. So now consider the pull-back of the cycle $z_1-z_2$ by $\phi$, that is
$$z_1'-z_2'=\phi^*(z_1-z_2)=\sum_{i=1}^n \phi^*j_{i*}[div(f_i)]$$
This pullback is well defined as $\phi$ is an isomorphism onto its image and $z_1-z_2$ is supported on $V$. Now all we have to prove is that
$$\phi^*j_{i*}[div(f_i)]=j_{i*}'\phi^*[div(f_i)]\;.$$
Now
$$\phi^*[div(f_i)]=[div(\phi^*(f_i))]$$ and
$$j_{i*}'[div(\phi^*(f_i))]=[div(N(\phi^*(f_i)))]$$
Here $N(\phi^*(f))$ is the determinant of the linear transformation given by $\phi^*(f)$. This is obtained as follows. Consider the map $j_i':\wt{C_i}\to \wt{D_i}$, where $\wt{D_i}$ is the image of $\wt{C_i}$ under $j_i'$ and it is an isomorphism. So we have $k(\wt{C_i})\cong k(\wt{D_i})$. Hence $N(\phi^*(f_i))$ is nothing but multiplication by $\phi^*(f_i)$.
Now since we also have $C_i\cong j_i(C_i):=D_i$ and $C_i\cong \wt{C_i}$,
$$N(\phi^*(f_i))=\phi^*(N(f_i))\;.$$
So we get that
$$div[N(\phi^*(f_i))]=div[\phi^*(N(f_i))]=\phi^*div[N(f_i)]=\phi^* j_{i*}[div(f_i)]\;.$$
Therefore we can conclude that $z_1'-z_2'$ is rationally equivalent to zero on $\Sym^g C$ and it is supported on $V'$.
\end{proof}

Now we recall the following lemma from \cite{Voi} to make the above intersection product well defined on all of $J(C)$.

\begin{lemma}
\label{lemma1}
Let $Y$ be a connected projective variety over an uncountable algebraically closed field $k$. Let $U\subset Y$ be the complement to a countable union of proper Zariski closed subsets of $Y$. Then any zero cycle $z$ on $Y$ is rationally equivalent in $Y$ to some $z'$ supported on $U$.
\end{lemma}

\begin{proof}
See \cite{Voi}[Fact 3.3].
\end{proof}

Let $j$ denote the embedding of $\sum_{i=1}^m C$ to $J(C)$. Similar to above we can define:
$$(z\times \sum_{i=1}^m C).(j\times \id)^*(\Gamma_1)$$
$$=(f_1\times f_1)_*((f_1\times f_1)^*(z\times \sum_{i=1}^m C).(f_1\times f_1)^*(j\times \id)^*(\Gamma_1))$$

Now we prove the following:

\begin{lemma}
\label{lemma3}
The homomorphism $\Gamma_{1*}j_*$ is induced by $(j\times \id)^*(\Gamma_1)$.
\end{lemma}

\begin{proof}
By the definition of $\Gamma_{1*}j_*(z)$ we have
$$\Gamma_{1*}j_*(z)=\pr_{2*}(j_*(z)\times \sum_{i=1}^m C. \Gamma_1)$$
which by the previous is equal to
$$\pr_{2*}(f_2\times f_1)_*((f_2\times f_1)^*(j_*(z)\times \sum_{i=1}^m C).(f_2\times f_1)^*\Gamma_1)$$
which is nothing but
$$f_{1*}(\pr_{2*}((f_2\times f_1)^*(j_*(z)\times \sum_{i=1}^m C).(f_2\times f_1)^*\Gamma_1))\;.$$
This because of the following commutative diagram:
$$
  \diagram
  \Sym^{g} C\times \Sym^{m}C\ar[dd]_-{f_2\times f_1} \ar[rr]^-{\pr_2} & & \Sym^{m} C \ar[dd]^-{f_1} \\ \\
  J(C)\times \sum_{i=1}^m C \ar[rr]^-{\pr_{2}} & & \sum_{i=1}^m C
  \enddiagram
  $$
On the other hand the homomorphism
$$(j\times \id)^*(\Gamma_1)_*(z)$$
is equal to
$$\pr_{2*}(z\times \sum_{i=1}^m C.(j\times \id)^*(\Gamma_1))$$
which by the previous is equal to
$$\pr_{2*}(f_1\times f_1)_*((f_1\times f_1)^*(z\times \sum_{i=1}^m C).(f_1\times f_1)^*(j\times \id)^*(\Gamma_1))\;.$$
This is nothing but
$$f_{1*}\pr_{2*}((f_1\times f_1)^*(z\times \sum_{i=1}^m C).(j'\times \id)^*(f_2\times f_1)^*\Gamma_1)$$
because of the following commutative diagram:

$$
  \diagram
  \Sym^{m}{C}\times \Sym^{m}C\ar[dd]_-{f_1\times f_1} \ar[rr]^-{j'\times \id} & & \Sym^{g} C\times \Sym^m C \ar[dd]^-{f_2\times f_1} \\ \\
  \sum_{i=1}^m C\times \sum_{i=1}^m C \ar[rr]^-{j\times \id} & & J(C)\times \sum_{i=1}^m C
  \enddiagram
  $$
Now we have
$$(f_1\times f_1)^*(z\times \sum_{i=1}^m C)=f_1^*(z)\times \Sym^{m}C$$
by the commutative diagram

$$
  \diagram
  \Sym^{m}C\times \Sym^{m}C\ar[dd]_-{} \ar[rr]^-{} & & \Sym^{m} C \ar[dd]^-{} \\ \\
  \sum_{i=1}^m C\times \sum_{i=1}^m C\ar[rr]^-{} & & \sum_{i=1}^m C
  \enddiagram
  $$
So by projection formula we get that
$$f_{1*}\pr_{2*}((f_1\times f_1)^*(z\times \sum_{i=1}^m C).(j'\times \id)^*(f_2\times f_1)^*\Gamma_1)$$
is equal to
$$f_{1*}\pr_{2*}((j'\times \id)_*(f_1\times f_1)^*(z\times \sum_{i=1}^m C).(f_2\times f_1)^*\Gamma_1)\;.$$
So all we need to prove that is
$$(j'\times \id)_*(f_1\times f_1)^*(z\times \sum_{i=1}^m C)=(f_2\times f_1)^*(j\times \id)_*(z\times \sum_{i=1}^m C)$$
So let $z$ be a zero cycle on $\sum_{i=1}^m C$, the above will follow from
$$j'_*f_1^*(z)=f_2^*j_*(z)\;.$$
By Lemma \ref{lemma1} we can assume that $z$ is supported on the complement of the exceptional locus of the map $\Sym^{m}C\to \sum_{i=1}^m C$.
Now supposing the above  this follows because of Lemma \ref{lemma2} and of the following: we have a Cartesian square

$$
  \diagram
  U'\ar[dd]_-{} \ar[rr]^-{} & & V' \ar[dd]^-{} \\ \\
  U\ar[rr]^-{} & & V
  \enddiagram
  $$
where $U'$ is Zariski open in $\Sym^m C$ and maps isomorphically onto $U$ in $\sum_{i=1}^m C$, and $V'$ is Zariski open in $\Sym^g C$ and maps birationally onto $V$ in $J(C)$.
\end{proof}

\medskip

\textit{Proof of theorem \ref{theorem1}}:

\medskip

Now we come back to the proof of the theorem \ref{theorem1}.

By the projection formula it follows that
$${\Gamma_1}_*j_*$$
is induced by $(j\times id)^*(\Gamma_1)$
where $j$ is the closed embedding of $\sum_{i=1}^m C$ into $J(C)$. Now we compute the cycle
$$(j\times id)^*(\Gamma_1)$$
that is nothing but the collection of rational equivalence classes of divisors
$$([D_1],[D_2])$$
such that $D_1$ is linearly equivalent to $\sum_{i=1}^{m} x_i-mp$ and $D_2$ is linearly equivalent to $\sum_{i=1}^{m}y_i-mp$, where
$$([x_1,\cdots,x_{m},p,p\cdots,p],[y_1,\cdots,y_{m}])\in \Gamma\;.$$
Therefore without loss of generality we can assume that elements in $(j\times id)^*(\Gamma_1)$ are classes of effective divisors on $C$ of the form
$$([x_1+\cdots+x_{m}+(g-m)p-gp],[y_1+\cdots+y_{m}+(g-m)p-gp])$$
such that
$$x_1+\cdots+x_m+(g-m )p=x_1'+\cdots+x_m'+(g-m)p+div(f)$$
and
$$y_1+\cdots+y_m+(g-m)p=y_1'+\cdots+y_m'+(g-m)p+div(g)$$
such that
$$([x'_1,\cdots,x'_{m},p,\cdots,p],[y'_1,\cdots,y'_{m}])\in \Gamma\;.$$
Now observe that the $g$-tuple $(x_1',\cdots,x_m',\cdots,p)$ projects down under $\pr$ to $(x'_1,\cdots,x_m')$ and the $g$-tuples $(x'_1,x'_2,\cdots,x'_{i-1},p,x'_i,\cdots,x'_m,p,\cdots,p)$ projects down to $(x'_1,\cdots,x'_{i-1},p,x'_i,\cdots,x'_m)$. Since
$$([x'_1,\cdots,x'_{m},p,\cdots,p],[y'_1,\cdots,y'_{m}])\in \Gamma$$
we have
$$[y'_1,\cdots,y'_m]=[x_1',\cdots,x_m']$$
or
$$[y_1',\cdots,y_m']=[p,\cdots,p,x'_{i+1},\cdots,x_m']$$
Then we get that either
$$x'_i=y'_i$$
for all $i$ or
$$y'_i=p$$
for some $i$. Since we are only considering the divisor classes and divisor of the above form represents a divisor class in $(j\times\id)^*\Gamma_1$, we get that $(j\times id)^*\Gamma_1$ is equal to
$$\Delta_{\sum_{i=1}^m C}+Y$$
where $\Delta_{\sum_{i=1}^m C}$ is the diagonal in $\sum_{i=1}^m C\times \sum_{i=1}^m C$ and $Y$ is supported on $\sum_{i=1}^m C\times \sum_{i=1}^{m-1}C$. The fact that the multiplicity of $\Delta_{\sum_{i=1}^m C}$ in $(j\times id)^*(\Gamma_1)$ is $1$ follows from the fact that $\Sym^g C$ maps surjectively and birationally onto $J(C)$, and the computations following \cite{Collino}. Also here the intersection product of a cycle of the form $z\times \sum_{i=1}^m C$ and $(j\times \id)^*(\Gamma_1)$ is well defined by Lemma \ref{lemma4}, \ref{lemma2}(because the sublemma \ref{sublemma1} holds  for $\sum_{i=1}^m C$).

Let's give details about the  computation of the cycle $\Gamma_{1*}j_*(z)$. Let $U'$ be the open subset of $\Sym^{m}C$  mapping birationally onto  $U$ in $\sum_{i=1}^m C$. Consider the inverse image of two cycles $\Delta_{\sum_{i=1}^m C}+Y$ on $\sum_{i=1}^m C$ and $z\times \sum_{i=1}^m C$ (where $z$ is a zero cycle on $\sum_{i=1}^m C$) by $f_1\times f_1$. Then by lemma \ref{lemma4},\ref{lemma2},\ref{lemma3}, \ref{lemma1}, we know that, on $\Sym^{m}C\times \Sym^{m}C$  the intersection of $(f_1\times f_1)^{*}(\Delta_{\sum_{i=1}^m C}+Y)$ and $(f_1\times f_1)^{*}(z\times \sum_{i=1}^m C)$ is well defined. So
$$(\Delta_{\sum_{i=1}^m C}+Y).(z\times \sum_{i=1}^m C)=(f_{1}\times f_1)_*((f_1\times f_1)^{*}(\Delta_{\sum_{i=1}^m C}+Y). (f_1\times f_1)^{*}(z\times \sum_{i=1}^m C))$$

We know that the natural map from $\Sym^m C$ to $\sum_{i=1}^m C$ is birational for all $m\leq g$, \cite{M}[Theorem 5.1]. Let $E$ be the exceptional locus of this map. Then the inverse image scheme $(f_1\times f_1)^{-1}(\Delta_{\sum_{i=1}^m C})$ consists of points $(x,y)$ in $\Sym^m C\times \Sym^m C$ such that
$$f(x)=f(y)$$
Since $f$ is a regular birational map, $f(x)=f(y)$ only when $x=y$ or $(x,y)$ is in the same linear system of divisors. In details, we have the fibers of the natural map $\Sym^m C\to \sum_{i=1}^m C$ are the linear systems of effective divisors linearly equivalent to an effective divisor $D$ of degree $m$. Write
$$x=\sum_{i=1}^m P_i,\quad y=\sum_{i=1}^m Q_i$$
since $f(x)=f(y)$ we have
$$\sum_{i=1}^m P_i=\sum_{i=1}^m Q_i+div(f)$$
for a non-zero rational function $f$ on $C$.
Therefore we have that $\sum_{i=1}^m P_i$ belongs to the linear system of the effective divisor $\sum_{i=1}^m Q_i$ and $\sum_{i=1}^m Q_i$ is contained in the linear system of  the effective divisor $\sum_{i=1}^m P_i$. The product of this two linear systems in contained in the product $E\times E$ because $E$ consists of all effective divisors of degree $m$ such that the corresponding linear system is positive dimensional,\cite{M}[Theorem 5.1]. Therefore we have
$$(f_1\times f_1)^{*}(\Delta_{\sum_{i=1}^m C}+Y)=\Delta_{\Sym^{m}C}+Z+Y'$$
where $Z$ is supported on  the exceptional locus $E\times E$ of the birational map $\Sym^m C\times \Sym^m C\to \sum_{i=1}^m C\times \sum_{i=1}^m C$ (where it is not birational) of the map $f_1\times f_1$ and $Y'$ is the inverse image of $Y$ under $f_1\times f_1$. Now for $(f_1\times f_1)^{*}(z\times \sum_{i=1}^m C)$,  we have
$$(f_1\times f_1)^{*}(z\times \sum_{i=1}^m C)=z'\times \Sym^{m}C$$
where $z'$ is a zero cycle.

So we have
$$(\Delta_{\Sym^{m}C}+Z+Y').(z'\times \Sym^{m}C)=z'\times z'+z'\times \Sym^{m}C.Y'$$
This is because of the following. First observe that $\Sym^m C$ is non-singular and hence the Chow moving lemma holds there. Therefore we can move the support of $z'$ away from $E$. Also this is possible by Lemma \ref{lemma1} to first move $z$ to general position and then obtain $z'$ as the pull-back of $z$ and the pull-back is well defined by Lemma \ref{lemma2}. Now the codimension of $z'\times \Sym^{m}C.Y'$ is $2m$, after removing the part coming from $E\times E$ from it. Let's give details on this. Following the argument as in Lemma \ref{lemma4} we have that
$$(j\times \id)^*(\Gamma_1)=(j'\times \id)^*(\Gamma)+Z$$
where $Z$ is supported on $E\times E$. Now following  the computation of \cite{Collino}, we have that
$$(j'\times \id)^*(\Gamma)=\Delta_{\Sym^m C}+Y'$$
where $Y'$ is of codimension $m$, this is because $\Sym^g C\times \Sym^m C$ is smooth hence $j'\times \id$ is a locally complete intersection morphism. Therefore the intersection of $Y'$ and $z'\times \Sym^m C$ is of codimension $2m$ by the Chow moving lemma for the non-singular variety $\Sym^m C\times \Sym^m C$.
Therefore we have
$$(\Delta_{\sum_{i=1}^m C}+Y).(z\times \sum_{i=1}^m C)=z\times z+Y.(z\times \sum_{i=1}^m C)$$
and it gives rise to a cycle of correct codimension.

So we have
$$\Gamma_{1*}j_*(z)=z+z'$$
if $z$ is in $U$ and $z'$ is supported on $\sum_{i=1}^{m-1}C$ and hence the above equality is true for all $z$ in $\sum_{i=1}^m C$.

So let
$$\rho:U''\to \sum_{i=1}^m C$$
be the open embedding of the complement of $\sum_{i=1}^{m-1}C$ in $\sum_{i=1}^m C$. Then we have
$$\rho^*{\Gamma_1}_*j_*(z)=\rho^*(z+z')=\rho^*(z)$$
where $z'$ is supported on $\sum_{i=1}^{m-1}C$. This follows since $(j\times id)^*(\Gamma_1)=\Delta_{\sum_{i=1}^m C}+Y$, where $Y$ is supported on $\sum_{i=1}^m C\times\sum_{i=1}^{m-1}C$ and from the previous computation. Now consider the following commutative diagram.
$$
  \xymatrix{
     \CH_0(\sum_{i=1}^{m-1}C) \ar[r]^-{j'_{*}} \ar[dd]_-{}
  &   \CH_0(\sum_{i=1}^m C) \ar[r]^-{\rho^{*}} \ar[dd]_-{j_{*}}
  & \CH_0(U'')  \ar[dd]_-{}  \
  \\ \\
   \CH_0(\sum_{i=1}^{m-1}C) \ar[r]^-{j''_*}
    & \CH_0(J(C)) \ar[r]^-{}
  & \CH_0(V'')
  }
$$
Here $U'',V''$ are complements of $\sum_{i=1}^{m-1}C$ in $\sum_{i=1}^{m}C,J(C)$ respectively. Now suppose that $j_*(z)=0$, then from the previous it follows that
$$\rho^*{\Gamma_1}_*j_*(z)=\rho^*(z)=0$$
by the localisation exact sequence it follows that there exists
$z''$ in $\sum_{i=1}^{m-1}C$ such that $j'_*(z'')=z$. Now consider the induction hypothesis for all $n\leq m-1 $, the embedding $\sum_{i=1}^n C$ into $J(C)$ gives rise to an injective push-forward at the level of Chow groups. The induction starts as the embedding of a closed point in $J(C)$ gives rise to an injective push-forward homomorphism by the following reason: the point must be contained in a smooth projective curve in $J(C)$ and then the equality
$$np=div(f)$$
for $f$ a rational function on the curve and $p$ being the closed point would imply  that $n=0$. This is because the divisor of rational functions on a smooth projective curve are of degree zero.

Then by the commutativity of the above diagram it follows that
$$j''_*(z')=0$$
would imply $z'=0$,
since by induction hypothesis  the embedding of $\sum_{i=1}^{m-1}C$ into $J(C)$ induce an injection at the level of Chow group of zero cycles . So we get that $z'=0$ hence $z=0$. So $j_*$ is injective.

Now putting $m=g-1$, we get the desired result.
\end{proof}

Now we prove the following:

\begin{theorem}
\label{theorem2} Let $A$ be a principally polarized abelian variety embedded into the Jacobian $J(C)$ of a smooth projective curve $C$. Suppose that $A$ is a proper subvariety of $J(C)$. Let $\Th_C$ denote the theta divisor of $J(C)$. Pull it back to $A$ and denote the pull-back by $\Th_A$. Suppose that $\Sym^{g-i}C\cap A'$ is smooth  for all $i\geq 0$ and suppose also that $C\cap A'$ is  irreducible, where $A'$ is the inverse image of $A$, under the map $\Sym^g C\to J(C)$. Then for the closed embedding $j:\Th_A\to A$,  $j_*$ has torsion kernel at the level of Chow groups of zero cycles.
\end{theorem}
\begin{proof}
As previous we prove the injectivity by induction, that is we prove that for any $m\leq g-1$, the closed embedding $\sum_{i=1}^m C\cap A\to A$ induces an injection at the level of Chow groups of zero cycles.
 We have a natural correspondence as previous, $\Gamma$ given by $\pi_g\times \pi_{m}(Graph(\pr))$ on $\Sym^g C\times \Sym^{m}C$, where $\pr$ is the projection from $C^g$ to $C^{m}$, $\pi_i$ is the natural morphism from $C^i$ to $\Sym^i C$. Consider the correspondence $\Gamma_1$ on $J(C)\times \sum_{i=1}^m C$ given by
$(f_2\times f_1)(\Gamma)$,
where $f_1,f_2$ are natural morphisms from $\Sym^{m}C,\Sym^g C$ to $\sum_{i=1}^m C,J(C)$ respectively. Consider the restriction of $\Gamma_1$ to $A\times (\sum_{i=1}^m C\cap A)$. Call it $\Gamma_1'$. Let $j$ denote the embedding of $\sum_{i=1}^m C\cap A$ into $A$. Then by projection formula it follows that
$${\Gamma_1'}_*j_*$$
is induced by $(j\times id)^*(\Gamma_1')$,
(this is similar to the arguments as in Lemma \ref{lemma3} and by the assumption that $A'$ is smooth).

Now we compute the cycle
$$(j\times id)^*(\Gamma_1')\;.$$

This is nothing but the restriction of the cycle $(j\times \id)^*(\Gamma_1)$ by the commutative diagram (here $j$ denote the embedding of $\sum_{i=1}^m C$ into $J(C)$).

$$
  \diagram
  \sum_{i=1}^m C\cap A\times \sum_{i=1}^m C\cap A\ar[dd]_-{} \ar[rr]^-{} & & A\times \sum_{i=1}^m C\cap A \ar[dd]^-{} \\ \\
  \sum_{i=1}^m C\times \sum_{i=1}^m C\ar[rr]^-{} & & J(C)\times \sum_{i=1}^m C
  \enddiagram
  $$

Therefore this cycle is the restriction of $\Delta_{\sum_{i=1}^m C}+Y$ (by Theorem \ref{theorem1}) to $(\sum_{i=1}^m C\cap A)\times (\sum_{i=1}^m C\cap A)$. Hence it is
$$d\Delta_{\sum_{i=1}^m C\cap A}+Y'$$
where $Y'$ is supported on $(\sum_{i=1}^m C\cap A)\times (\sum_{i=1}^{m-1}C\cap A)$.

Here $d$ is a positive natural number. Also here the intersection product for the cycles of form $z\times (\sum_{i=1}^m C\cap A)$ and $(j\times \id)^*(\Gamma_1')$ holds for $\sum_{i=1}^m C\cap A\times \sum_{i=1}^m C\cap A$ by the Lemmas' \ref{lemma1},\ref{lemma2},\ref{lemma3}. Let us give the details. Let us consider a  zero cycle $z$ on $\CH_0(\sum_{i=1}^m C\cap A)$. Then consider the cycles $z\times \sum_{i=1}^m C\cap A$ and $d\Delta_{\sum_{i=1}^m C\cap A}+Y$. Since the inverse image $\Sym^{m}C\cap A'$ is smooth, Chow moving lemma holds on the self product $\Sym^{m}C\cap A'\times \Sym^{m}C\cap A'$. Let $f_1$ be the map from $A'\to A$. Then we define the product (as in the beginning of Theorem \ref{theorem1})
$$(z\times \sum_{i=1}^m C\cap A).(d\Delta_{\sum_{i=1}^m C\cap A}+Y)$$
to be
$$:=(f_1\times f_1)_*((f_1\times f_1)^{*}(z\times \sum_{i=1}^m C\cap A).(f_1\times f_1)^{*}(d\Delta_{\sum_{i=1}^m C\cap A}+Y'))$$

Now
$$(f_1\times f_1)^*(d\Delta_{\sum_{i=1}^m C\cap A})=d\Delta_{\Sym^m C\cap A'}+Z'$$
where $Z'$ is supported on the product $\Sym^m C\cap A'\cap E\times \Sym^m C\cap A'\cap E$, where $E$ is the exceptional locus of the birational map $\Sym^g C\to J(C)$.
Let $z'$ be a zero cycle on $\Sym^m C\cap A'$.
Then following the computation as in Theorem \ref{theorem1}, we have that
$$z'.E=0$$
because we can move the support of the zero cycle $z'$ away from $Z'$ in $\Sym^m C\cap A'$. Therefore we get that
$$(z'\times \Sym^m C\cap A').(d\Delta_{\Sym^m C\cap A'}+Z'+Y')$$
is equal to
$$z'\times z'+z'\times Y''$$
where $Y''$ is supported on $\Sym^m C\cap A'\times \Sym^{m-1}C\cap A'$. It is the cycle $(f_1\times f_1)^* Y'$. Then arguing as in Theorem \ref{theorem1} we have
$$(z\times \sum_{i=1}^m C\cap A).(d\Delta_{\sum_{i=1}^m \cap A}+Y')=z\times z+(z\times \sum_{i=1}^m C\cap A).Y'$$
Moreover the cycle $z\times \sum_{i=1}^m C\cap A.Y'$ is of correct codimension.

 Therefore we have
$$\Gamma_{1*}'j_*(z)=z+z'$$
for all $z$ in  $\sum_{i=1}^m C\cap A$, and $z'$ supported on $\sum_{i=1}^{m-1}C\cap A$.

So let
$$\rho:U\to \sum_{i=1}^m C\cap A$$
be the open embedding of the complement of $\sum_{i=1}^{m-1}C\cap A$ in $\sum_{i=1}^m C\cap A$. Then we have
$$\rho^*{\Gamma_1'}_*j_*(z)=\rho^*(dz+z_1)=d\rho^*(z)$$
where $z_1$ is supported on $\sum_{i=1}^{m-1}C\cap A$. This follows since $(j\times id)^*(\Gamma_1')=d\Delta_{\sum_{i=1}^m C\cap A}+Y'$, where $Y'$ is supported on $\sum_{i=1}^m C\cap A\times\sum_{i=1}^{m-1}C\cap A$. Now consider the following commutative diagram.
$$
  \xymatrix{
     \CH_0(\sum_{i=1}^{m-1}C\cap A) \ar[r]^-{j'_{*}} \ar[dd]_-{}
  &   \CH_0(\sum_{i=1}^m\cap A) \ar[r]^-{\rho^{*}} \ar[dd]_-{j_{*}}
  & \CH_0(U)  \ar[dd]_-{}  \
  \\ \\
   \CH_0(\sum_{i=1}^{m-1}C\cap A) \ar[r]^-{j''_*}
    & \CH_0(A) \ar[r]^-{}
  & \CH_0(V)
  }
$$
Here $U,V$ are complements of $\sum_{i=1}^{m-1}C\cap A$ in $\sum_{i=1}^m C\cap A,A$ respectively. Now suppose that $j_*(z)=0$, then from the previous it follows that
$$\rho^*{\Gamma_1'}_*j_*(z)=\rho^*(dz)=0$$
by the localisation exact sequence it follows that there exists
$z'$ supported on $\sum_{i=1}^{m-1}C\cap A$ such that $j'_*(z')=dz$. Now assume the induction hypothesis: that is for all $n\leq m-1$ we have the closed embedding $\sum_{i=1}^n C\cap A\to A$ induces a push-forward at the level of Chow groups of zero cycles whose kernel is torsion. The induction starts with the assumption that $C\cap A$ is irreducible and hence a closed point (because $A$ is a proper abelian subvariety of $J(C)$), therefore arguments as in Theorem \ref{theorem1} would prove that the embedding of this point into $A$, induces an injection at the level of Chow groups of zero cycles.

Now by the commutativity of the previous diagram  it follows that
$$j''_*(z')=0$$
By the induction hypothesis we have  that the closed embedding $\sum_{i=1}^{m-1}C\cap A\to A$, induces a push-forward at the level of Chow groups of zero cycles whose kernel is torsion.  So we get that $d'z'=0$ hence $dd'z=0$. So $j_*$ has torsion kernel.

Now putting $m=g-1$ we get the desired result.

\end{proof}

\begin{theorem}
\label{theorem4}
Let $A$ be a principally polarized abelian variety embedded into the Jacobian $J(C)$ of a smooth projective curve $C$. Suppose that $A$ is a proper subvariety of $J(C)$. Let $\Th_C$ denote the theta divisor of $J(C)$. Pull it back to $A$ and denote the pull-back by $\Th_A$. Suppose that $\Sym^{g-i}C\cap A'$ is singular  for some $i\geq 0$ and suppose also that $C\cap A'$ is  irreducible, where $A'$ is the inverse image of $A$, under the map $\Sym^g C\to J(C)$. Then for the closed embedding $j:\Th_A\to A$,  $j_*$ has torsion kernel at the level of Chow groups of zero cycles.
\end{theorem}

\begin{proof}
The proof here goes similarly as in Theorem \ref{theorem2} but the only matter of concern is the injectivity of $\CH_0(\sum_{l=1}^{m}C\cap A)\to \CH_0(A)$ for $m\leq g-1$, where $\Sym^m C\cap A'$ is singular. So suppose that $\Sym^m C\cap A'$ is singular. Let $\wt{\Sym^{m}C\cap A'}$ is the resolution of singularities of it.  Then considering embedding $j':\Sym^m C\cap A'\to  A'$ and the induced embedding $j:\sum_{l=1}^m C\cap A\to A$ we have (following the arguments of Lemma \ref{lemma3}) $\Gamma'_{1*}j_*$ is induced by $(j\times \id)^*(\Gamma'_1)$. Following Theorem \ref{theorem2} we have
$$(j\times \id)^*(\Gamma'_1)=d\Delta_{\sum_{l=1}^m C \cap A}+Y$$
where $\Delta_V$ is the diagonal of $V$ and $Y$ is supported on $\sum_{l=1}^m C\cap A\times \sum_{l=1}^{m-1}C\cap A$. Let's put $V=\sum_{l=1}^m C\cap A$. Now as before consider  the cycles $z\times V$ and $d\Delta_{V}+Y$. Since the variety $\wt{\Sym^{m}C\cap A'}$ is smooth, Chow moving lemma holds on the self product $\wt{\Sym^m C\cap A'}\times \wt{\Sym^{m}C\cap A'}$ and we have the inverse images (as defined in the proof of theorem \ref{theorem1},under the map from $\wt{\Sym^{m}C\cap A'}\times \wt{\Sym^{m}C\cap A'}$ to $\sum_{l=1}^m C\cap A\times \sum_{l=1}^m C\cap A$) of $z\times V$ and $d\Delta_{V}+Y$ intersect properly on the product. Let $f_1$ be the map from $\wt{\Sym^m C\cap A'}\to \sum_{l=1}^m C\cap A$. Then we define the product
$$(z\times V).(d\Delta_{V}+Y)$$
to be
$$:=(f_1\times f_1)_*((f_1\times f_1)^{*}(z\times V).(f_1\times f_1)^{*}(d\Delta_{V}+Y))\;.$$

Note that
$$(f_1\times f_1)^{*}(d\Delta_{V}+Y)=\Delta_{\wt{\Sym^{m}C\cap A'}}+{Z'}+ Z''+Y'$$
where $Z'$ is supported on the inverse image of  $E'\times E'$ where, $E'=E\cap \Sym^{m}C\cap A'$ under the map $\wt{\Sym^m C\cap A'}\times \wt{\Sym^m C\cap A'}\to \Sym^m C\cap A'\times \Sym^m C\cap A'$. The cycle $Z''$ is supported on $E''\times E''$, $E''$ being the exceptional divisor associated with the resolution of singularities of $\Sym^m C\cap A'$. Suppose that $(f_1\times f_1)^{*}(z\times V)=z'\times \wt{\Sym^{m}C\cap A'}$, where $z'$ is a zero cycle. Then the intersection:
$$(f_1\times f_1)^{*}(z\times V).(f_1\times f_1)^{*}(d\Delta_{V}+Y)$$
$$=z'\times z'+(z'\times \wt{\Sym^{m}C\cap A'}).({Z'}+Z'')+Y'.(z'\times \Sym^{m-1}C\cap A')\;.$$
Here  $Y'$ is supported on $(\Sym^{m}C\cap A')\times(\Sym^{m-1}C\cap A')$. Then moving the support away from $E''$ we have $z.E''=0, z.\hat{E'}=0$ ($\hat{E'}$ is the inverse image of $E'$ under the map $\wt{\Sym^m C\cap A'}\to \Sym^m C\cap A'$). Hence the above intersection comes to
$$z'\times z'+Y'.(z'\times \wt{\Sym^{m}C\cap A'})$$
similar reasons as in Theorem \ref{theorem1} we have the second part $Y'.(z'\times \wt{\Sym^{m}C\cap A'})$ is of correct codimension. Therefore
$$(f_1\times f_1)_*(z'\times z'+Y'.(z'\times \Sym^{m-1}C\cap A'))$$
is well defined.

Therefore we have
$$\Gamma_{1*}'j_*(z)=z+z'$$
for all $z$  in $\sum_{l=1}^m C\cap A$ and $z'$ supported on $\sum_{l=1}^{m-1}C\cap A$.

So let
$$\rho:U\to \sum_{l=1}^m C\cap A$$
be the open embedding of the complement of $\sum_{l=1}^{m-1}C\cap A$ in $\sum_{l=1}^m C\cap A$. Then we have
$$\rho^*{\Gamma_1'}_*j_*(z)=\rho^*(dz+z_1)=d\rho^*(z)$$
where $z_1$ is supported on $\sum_{l=1}^{m-1}C\cap A$. This follows since $(j\times id)^*(\Gamma_1')=d\Delta_{\sum_{i=1}^m C\cap A}+Y$, where $Y$ is supported on $\sum_{l=1}^m C\cap A\times\sum_{l=1}^{m-1}C\cap A$.

Now consider the following commutative diagram.
$$
  \xymatrix{
     \CH_0(\sum_{l=1}^{m-1}C\cap A) \ar[r]^-{j'_{*}} \ar[dd]_-{}
  &   \CH_0(\sum_{l=1}^m C\cap A) \ar[r]^-{\rho^{*}} \ar[dd]_-{j_{*}}
  & \CH_0(U)  \ar[dd]_-{}  \
  \\ \\
   \CH_0(\sum_{l=1}^{m-1}C\cap A) \ar[r]^-{j''_*}
    & \CH_0(A) \ar[r]^-{}
  & \CH_0(V)
  }
$$

Here $U,V$ are complements of $\sum_{l=1}^{m-1}C\cap A$ in $\sum_{l=1}^m C\cap A,A$ respectively. Now suppose that $j_*(z)=0$, then from the previous it follows that
$$\rho^*{\Gamma_1'}_*j_*(z)=\rho^*(dz)=0$$
by the localisation exact sequence it follows that there exists
$z'$ supported on $\sum_{l=1}^{m-1}C\cap A$ such that $j'_*(z')=dz$. Now assume by induction hypothesis that for all $n\leq m-1$ we have the closed embedding $\sum_{l=1}^n C\cap A\to A$ induces a push-forward at the level of Chow groups of zero cycles whose kernel is torsion. The induction starts with the assumption that $C\cap A$ is irreducible and hence a closed point (because $A$ is a proper abelian subvariety of $J(C)$), therefore arguments as in Theorem \ref{theorem1} would prove that the embedding of this point into $A$, induces an injection at the level of Chow groups of zero cycles.

Now by the commutativity of the previous diagram  it follows that
$$j''_*(z')=0$$
By the induction hypothesis we have  that the closed embedding $\sum_{l=1}^{n}C\cap A\to A$, for $n\leq m-1$ induces a push-forward at the level of Chow groups of zero cycles whose kernel is torsion.  So we get that $d'z'=0$ hence $dd'z=0$. So $j_*$ has torsion kernel.

\end{proof}

The previous theorem gives rise to the following corollary:

\begin{corollary}
Let $\wt{C}\to C$ be an unramified double cover of smooth projective curves such that the genus of $C$ is atleast $2$. Consider the Prym variety associated to this double cover, denote by $P(\wt{C}/C)$. Consider the embedding of $P(\wt{C}/C)$ into $J(\wt{C})$. Let $\Th'$ be the pullback of the theta divisor on $J(\wt{C})$ to $P(\wt{C}/C)$. Then the closed embedding $\Th'\to P(\wt{C}/C)$ induces a pushforward homomorphism at the level of Chow groups which has torsion kernel.
\end{corollary}

\begin{proof}
We need to check that the Prym variety and the restriction of the theta divisor on $J(\wt{C})$ satisfies the assumption of Theorem \ref{theorem4}.

Let $g$ be the genus of $C$. So by Riemann-Hurwitz formula the genus of $\wt{C}$ is $2g-1$. We have the following commutative square.

$$
  \diagram
  \Sym^{2g-1} \wt{C}\ar[dd]_-{\wt{q}} \ar[rr]^-{} & & \Sym^{2g-1} C \ar[dd]^-{q_C} \\ \\
  J(\wt{C}) \ar[rr]^-{} & & J(C)
  \enddiagram
  $$
 Then first of all the Prym variety is the image under $\wt{q}$ of the double cover $P'$ of a projective space $\PR^{g-1}$. This follows from the Riemann-Roch and the very definition of the Prym variety. Now consider the intersection of $\Sym^{2g-i}\wt{C}$ with $P'$, where $i\geq 2$. This intersection  may be singular for a general copy of $\Sym^{2g-i}\wt{C}$ in $\Sym^{2g-1}\wt{C}$ but they fits into a family of algebraic varities in the following way. Consider the family
 $$\bcU:=\{([x_1,\cdots,x_{2g-2},x_{2g-1}],p)\in P'\times \wt{C}|p\in [x_1,\cdots,x_{2g-1}]\}$$
 and the projection from $\bcU$ to $\wt{C}$. Then $\bcU$ is a family over $\wt{C}$ with fibers $\Sym^{2g-2}\wt{C_p}\cap P'$ over $\wt{C}$, for $p\in \wt{C}$ and $\wt{C_p}\cong \wt{C}$. Now $P'$ is an unramified double cover of the projective space  $\PR^{g-1}$. Therefore $P'$ cannot be non-singular because if it is non-singular then $P'$ is an \'etale cover of $\PR^{g-1}$ making it isomorphic to $\PR^{g-1}$. Then $\PR^{g-1}$ maps to an abelian variety so the map should be constant making the Prym variety trivial, which is not the case.

Therefore $\bcU$ is a singular variety and its tangent space at a point is a hyperplane in the product
  of tangent spaces
  $$T_{[x_1,\cdots,x_{2g-1}]}P'\times T_p(\wt{C})\;.$$
So we resolve the singularity of $\bcU$ and construct a non-singular variety $\bcU'$. Consider the map from $\bcU'\to \wt C$.

Following the theorem 2.27 in \cite{Shafarevich}, over a Zariski open, non-empty subset $V$ of $\wt{C}$, the map $\bcU'_V:=\bcU'\times _{\wt{C}}V\to \wt{V}$ is a submersion, hence the fibers $\bcU'_p$ over $p\in V$ are non-singular subvarities of $\bcU'_V\to V$.

Hence  for a general $p$ in $\wt{C}$, $\bcU'_p$ is smooth which is nothing but a resolution of singularities of $\Sym^{2g-2}\wt C_p\cap P'$.

 Similarly we can prove that for a general $\Sym^{2g-i}\wt{C_p}\cap P'$ its resolution of singularity also fits into a family of algebraic varieties.

 Also considering the family $$\bcU':=\{([x_1,x_2,\cdots,x_{2g-1}],p)\in P'\times \wt{C}|x_2=x_3=\cdots=x_{2g-1}=p\}$$
the fiber of $\bcU'\to \wt{C}$ is isomorphic to $\wt{C_p}\cap P'$ for any point $p\in \wt{C}$, here $\wt{C_p}\cong \wt{C}$. The collection of points $x$ in $\wt{C_p}\cap P'$
satisfies that
$$(i(x)-i(p))=(-x+p)$$
or $(i(x)+x)=(i(p)+p)$ on $J(\wt{C})$ for a fixed point $p\in \wt{C}$. Therefore on $J(C)$ we have
$$2x'=2p'\;.$$
Here $x',p'$ are the images of $x,p$ under the quotient map from $\wt{C}$ to $C$.

Conversely suppose that we have for a point $x'$ in ${C}$ such that
$$2(x'-p')=0$$
Consider the pre-images $x,i(x)$ of $x'$ and $p,i(p)$ of $p$ in $\wt{C}$. Then
$$(x-p)+(i(x)-i(p))$$
is mapped to
$$2(x'-p')=0\;.$$
Therefore we have $x, i(x)\in \wt{C_p}\cap P'$. Therefore
$\wt{C_p}\cap P'/i$ is in bijection with $\bcL((4g-2)p')\cap C$, where $\bcL((4g-2)p')$ is the linear system associated to the cycle $(4g-2)p'$. By Riemann-Roch this linear system is of dimension $3g-1$. Therefore we have $\Sym^{4g-2}C$ is a projective bundle over $J(C)$. Now consider the family
$$\bcU=\{((4g-2)p,q)|q\in \bcL((4g-2)p)\}\subset \Sym^{4g-2}C\times C$$
over $\Sym^{4g-2}C$. Then for a $p$ in $C$ realised as an element of $\Sym^{4g-2}C$, the inverse image under the above projection is $C\cap \bcL((4g-2)p)$. Now the above family can be realised as the pull-back  to $C$ of the projective bundle  on $\Sym^{4g-2}C$ given by
$$\{(\sum_{i=1}^{4g-2}P_i,\sum_{j=1}^{4g-2}Q_j)|\sum_{j=1}^{4g-2}Q_j\in \bcL(\sum_{i=1}^{4g-2}P_i)\}\subset \Sym^{4g-2}C\times \Sym^{4g-2}C\;.$$
Therefore $\bcU$ is irreducible, hence for a general $p$, $\bcL((4g-2)p)\cap C$ is irreducible, \cite{Shafarevich}[theorem 2.26].

Hence it is a point, otherwise we have $C\subset \bcL((4g-2)p')$, which in turn gives us that all points of $\wt{C}$ are $i$-antiinvariant. Therefore $C$ is rational contradicting the assumption that genus of $C$ is atleast $2$. Since $\wt{C_p}\cap P'/i$ is in bijection with $\bcL((4g-2)p')\cap C$, we have $\wt{C_p}\cap P'$ consists of two points say $x,i(x)$. We have
$$i(x)-i(p)=-x+p\;.$$
We have to prove that the pushforward from $\CH_0(\{x-p,i(x-p)\})\cong \ZZ\oplus \ZZ$ to $\CH_0(P(\wt{C}/C))$ is injective. Consider $z$ in the  Chow group $\CH_0({x-p,i(x-p)})$, such that $j_*(z)$ is rationally equivalent to zero on $P(\wt{C}/C)$. Here $j$ is the closed embedding of the two points $\{x-p,i(x-p)\}$ into $P(\wt{C}/C)$. Then by the definition of rational equivalence there exist finitely many smooth projective curves $X_i$ in $P(\wt{C}/C)$ and rational functions $f_i$ on $X_i$ such that
$$j_*(z)=\sum_i div(f_i)$$
Since $div(f_i)$ has degree zero for each $i$, we have $j_*(z)$ is degree zero, hence so is $z$. So $z$ will look like
$$n(x-p)-ni(x-p)\;.$$
Now from the previous we have that the zero cycle
$$n(x-p)+ni(x-p)$$
maps to zero under the albanese map. So combining the previous two we have
$$2n(x-p)$$
maps to zero under the albanese map. So the albanese image of $x-p$ is $2n$-torsion. Therefore by the Roitman's torsion theorem the zero cycle $x-p$ is a $2n$-torsion. Therefore the zero cycle of $i(x-p)$ is also $2n$-torsion. But then by the definition of rational equivalence it follows that the degree of the cycles
$$2n(x-p)$$
and
$$2ni(x-p)$$
are zero. Which means that $n=0$. Therefore $z=0$ in $\CH_0(\{x-p,i(x-p)\})$.

 Hence we have the assumption of the Theorem  \ref{theorem4} is satisfied, whence the conclusion follows.
\end{proof}

Now we prove that if we blow up $J(C)$ at finitely many points and denote the blow up by $\wt{J(C)}$ and let $\wt{\Th_C}$ denote the strict transform of $\Th_C$, then the closed embedding of $\wt{\Th_C}$ into $\wt{J(C)}$ induces injective push-forward homomorphism at the level of Chow groups.

\begin{theorem}
\label{theorem3}
Let $\wt{J(C)}$ be the blow up of $J(C)$ at some non-singular subvariety $Z$ whose inverse image is $E$. Let $Z'$ be the inverse image of $Z$ under the map $\Sym^g C\to J(C)$. Consider the blow-up of $\Sym^g C$ along $Z'$, denote it by $\wt{\Sym^g C}$. Suppose that the blow-up of $\Sym^{g-i} C$ along $Z'\cap \Sym^{g-i}C$ (that is the strict transform of $\Sym^{g-i}C$ under the above blow-up) is smooth.  Let $\wt{\Th_C}$ denote the strict transform of $\Th_C$ in $\wt{J(C)}$. Then the closed embedding of $\wt{\Th_C}$ into $\wt{J(C)}$ induces injective push-forward homomorphism at the level of Chow groups of zero cycles with $\QQ$ coefficients.
\end{theorem}

\begin{proof}
We prove it by induction. That is we prove that the closed embedding of the strict transform $\wt{\sum_{i=1}^m C}$ of $\sum_{i=1}^m C$ into $\wt{J(C)}$ induces injective push-forward at the level of Chow group of zero cycles for all $m\leq g-1$.
Let us consider $\pi$ to be the morphism from $\wt{J(C)}$ to $J(C)$. Consider the correspondence $(\pi\times \pi')^*(\Gamma_1)$, where $\pi'$ is the restriction of $\pi$ to $\wt{\sum_{i=1}^m C}$. Call this correspondence $\Gamma'$. Then $\Gamma'_*\wt{j}_*$ is induced by $(\wt{j}\times id)^*\Gamma'$ (following Lemma \ref{lemma3}), where $\wt{j}$ is the closed embedding of $\wt{\sum_{i=1}^m C}$ into $\wt{J(C)}$. Consider the commutative square.
$$
  \diagram
   \wt{\sum_{i=1}^m C}\times \wt{\sum_{i=1}^m C}\ar[dd]_-{\pi'\times \pi'} \ar[rr]^-{\wt{j}\times id} & & \wt{J(C)}\times \wt{J(C)} \ar[dd]^-{\pi\times \pi} \\ \\
  \sum_{i=1}^m C\times \sum_{i=1}^m C \ar[rr]^-{j\times id} & & J(C)\times J(C)
  \enddiagram
  $$

This gives us that
$$(\wt{j}\times \id)^*\Gamma'=(\pi'\times \pi')^*(j\times id)^*\Gamma_1=(\pi'\times \pi')^*(\Delta_{\sum_{i=1}^m C}+Y)$$
where $Y$ is supported on $\sum_{i=1}^m C\times \sum_{i=1}^{m-1}C$. Now
$$(\pi'\times \pi')^*(\Delta_{\sum_{i=1}^m C})=\Delta_{\wt{\sum_{i=1}^m C}}+V$$
where $E$ is the exceptional locus of $\pi$ and $V$ is supported on $(E\cap \wt{\sum_{i=1}^m C})\times (E\cap \wt{\sum_{i=1}^m C})$.

Let $\wt{\Sym^{m}C}$ denote the fiber product of $\Sym^{m}C$ and $\wt{\sum_{i=1}^m C}$ over $J(C)$. Since $\Sym^{m}C$ is birational to $\sum_{i=1}^m C$, we have that $\wt{\Sym^{m}C}$ is birational to $\wt{\sum_{i=1}^m C}$. Now $\wt{\Sym^{m}C}$ is a blow-up of $\Sym^{m}C$ along $Z'\cap \Sym^{m}C$, which we assume to be smooth. Therefore the Chow moving lemma holds on $\wt{\Sym^{m}C}\times \wt{\Sym^{m}C}$, therefore arguing as in Theorem \ref{theorem1}, we can define the  intersection of two cycles on $\wt{\sum_{i=1}^m C}\times \wt{\sum_{i=1}^m C}$. So let us consider the cycles $\Delta_{\wt{\sum_{i=1}^m C}}+V+Y$ and $z\times \wt{\sum_{i=1}^m C}$. We consider the pullback of them by $f\times f$, where $f$ is the birational map from $\wt{\Sym^{m}C}\to \wt{\sum_{i=1}^m C}$. Then first note that
$$(f\times f)^*(\Delta_{\wt{\sum_{i=1}^m C}}+V+Y)=\Delta_{\wt{\Sym^{m}C}}+(f\times f)^*(V)+V'+Y'$$
where $V'$ is supported on the exceptional divisor $E'$  of the map $\wt{\Sym^{m}C}\to \wt{\sum_{i=1}^m C}$ and $Y'$ is supported on $\wt{\Sym^{m}C}\times \wt{\Sym^{m-1}C}$.  Then as before following computation as in Theorem \ref{theorem4} for a zero cycle $z$ on $\wt{\sum_{i=1}^m C }$ we have
$$(f\times f)_*((f\times f)^*(\Delta_{\wt{\sum_{i=1}^m C}}+V+Y).(f\times f)^*(z\times \wt{\sum_{i=1}^m C}))=z\times z+Y.(z\times \wt{\sum_{i=1}^m C})\;.$$

Consider the following commutative diagram.
$$
  \xymatrix{
     \CH_0(A) \ar[r]^-{\wt{j'}_{*}} \ar[dd]_-{}
  &   \CH_0(\wt{\sum_{i=1}^m C}) \ar[r]^-{\rho_0^{*}} \ar[dd]_-{\wt{j}_{*}}
  & \CH_0(U)  \ar[dd]_-{}  \
  \\ \\
   \CH_0(A) \ar[r]^-{\wt{j''}_*}
    & \CH_0(\wt J(C)) \ar[r]^-{}
  & \CH_0(V)
  }
$$
Here $A= \wt{\sum_{i=1}^{m-1}C}$.
Now suppose that $\wt{j}_*(z)=0$. By the previous computation we get that $\rho^*(z)=0$, so by the localisation exact sequence we get that there exists $z'$ in $\CH_*(A)$ such that $\wt{j'}_*(z')=z$.

Assume that for all $n\leq m-1$, we have the map $\CH_0(\wt{\sum_{i=1}^n C})\to \CH_0(J(C))$ is injective.  The induction starts because $\wt{p}$ for a closed point on $J(C)$, is either the point itself or it is empty, so argument similar to Theorem \ref{theorem1} applies here.

By the induction  hypothesis $\CH_*(A)\to \CH_*(\wt{J(C)})$ is injective. So we get that $z'=0$, hence $z=0$ giving $\wt{j_*}$ injective.

Putting $m=g-1$ we get the result.

\end{proof}

Now let $A$ be an abelian surface which is embedded in some $J(C)$. Let $i$ denote the involution of $A$. Then $i$ has $16$ fixed points. We blow up $A$ along these fixed points. Then we get $\wt{A}$ on which we have an induced involution, call it $i$. Let $\wt{\Th_A}$ denote the strict transform of $\Th_A$ in $\wt{A}$. Suppose that the abelian surface $A$ intersects the curve $C$ at one point. Then the techniques  as in Theorem \ref{theorem3}   tells us the following.

\begin{theorem}
Let $A$ be as above. Then the closed embedding of $\wt{\Th_A}$ into $\wt{A}$ induces injective push-forward homomorphism at the level of Chow groups of zero cycles.
\end{theorem}

Now $\wt{A}/i$ is the Kummer's K3 surface associated to $A$. Suppose that we choose $\Th_C$ such that it is $i$ invariant. Then $\wt{\Th_A}$ will be $i$ invariant. The above theorem gives us:

\begin{theorem}
Let $A$ be as above then the closed embedding of $\wt{\Th_A}/i$ into $\wt{A}/i$ induces injective push-forward homomorphism at the level of Chow groups of zero cycles with $\QQ$-coefficients.
\end{theorem}

Note that all these techniques can be repeated if we consider the group of algebraic cycles modulo algebraic equivalence. Then therefore the closed embedding $\wt{\Th_A}/i$ into $\wt{A}/i$ induces injection at the level of zero cycles modulo algebraic equivalence.

\section{Application of the above result}

Let us consider a smooth, cubic threefold $X$ embedded in $\PR^4$. Let us consider a line $l$ in $X$ and project from that line $l$ onto $\PR^2$. Let us consider the discriminant curve $C$ embedded in $\PR^2$. Let us consider the \'etale double covering of this curve $C$ say $\wt{C}$ inside $F(X)$, the Fano variety of lines of $X$. Let us consider the Prym variety of this double covering and the theta divisor of the principally polarized Prym variety. Call it $P(\wt{C}/C)$ and the theta divisor as $\Theta$. Let us consider the universal line correspondence $L$ on $X\times F(X)$ and consider the correspondence induced on $X\times P(\wt{C}/C)$ by this, say L. Then we consider the $L_*$ from $A_0(P(\wt{C}/C))$ to $A_1(X)$. Consider the pull back of $\Theta$ to $X\times P(\wt{C}/C)$ and intersect it with the correspondence $L$ and then push down the intersection  to $X$. Call the support of the cycle obtained this way as $L(\Theta)$. Consider the embedding
$$j':L(\Theta)\to X$$
and consider the homomorphism induced at the level of algebraically trivial 1-cycles
$$j'_*:A_1(L(\Theta))\to A_1(X)\;.$$
From the equality as in \ref{theorem4}
$$(j\times \id)^*\Gamma_1=d\Delta_{\Theta}+Y$$
where $Y$ is supported on $\Theta\times\sum_{i=1}^{g-2}\wt{C}\cap P(\wt{C}/C)$, $g$ the genus of $\wt{C}$. Then same calculation as in \ref{theorem4}, proves that the homomorphism $j'_*$ has torsion kernel.

Now with the help of this we prove that:

\begin{theorem}
 There does not exist a universal codimension 2 cycle $Z$ on $ P(\wt{C}/C)\times X$ such that $a\mapsto AJ(Z_*(a))=Id$
\end{theorem}

\begin{proof}
Suppose that there exists such a cycle $Z$ such that $AJ(Z_*(a))=Id$. Then we have $Z_*$ factoring through $A_0(P(\wt{C}/C))\to A_1(X)$ (this is by the very definition of $Z_*$) and the universal-ness tells us that the homomorphism from $P(\wt{C}/C)$ to $A_1(X)$ is injective. Consider the square
$$
  \diagram
   \Theta\ar[dd] \ar[rr] & & P(\wt{C}/C) \ar[dd] \\ \\
  A_1(L(\Theta)) \ar[rr]& & A_1(X)
  \enddiagram
  $$

Now considering a very general cubic threefold $X$, we have the abelian variety $A_1(X)\cong P(\wt{C}/C)$ is simple. Consider the desingularization of $L(\Theta)$, call it $\widetilde{L(\Theta)}$. Therefore the image of $A_1(\widetilde{L(\Theta)})\cong \Pic^0(\widetilde{L(\Theta)})$  is either $\{0\}$ or all of $A_1(X)$. It can't be zero. Then it will follow by the existence of the universal codimension two cycle that, $\Theta$ is embedded into the trivial abelian variety , which implies $\Theta$ is finite and hence $P(\wt{C}/C)$ is one dimensional, which is not true.

Now suppose that the image of $A_1(\widetilde{L(\Theta)})$ in $A_1(X)$ is all of $A_1(X)$. The singularity of the theta divisor of the Prym variety (which is identified with $A_1(X)$) is of zero dimension and consists of a unique singular point \cite{H}[Corollary 4.6]. The map from $A_1(\widetilde{L(\Theta)})$ to $P(\wt{C}/C)$ is a map of principally polarized abelian varieties. Observe that the theta divisor $\Theta$ is embedded  in $A_1(\widetilde{L(\Theta)})$ by the existence of the universal codimension two cycle. Consider the pull-back of $\Theta$ to $A_1(\widetilde{L(\Theta)})$, say $f^*(\Theta)$, where $f$ is the isogeny from $A_1(\widetilde{L(\Theta)})$ to $A_1(X)\cong P(\wt{C}/C)$. Since by the above diagram $\Theta$ is mapped to $\Theta$ by the composition of the inclusion of $\Theta$ and $AJ\circ Z_*$, we have
$$\Theta\subset f^*(\Theta)\;.$$
We Prove that $n\Theta=f^*(\Theta)$ for some integer $n>0$. Suppose that
$$f^*(\Theta)=\Theta+\Theta_1+\cdots+\Theta_n$$
Suppose $g$ is the map from $\Theta$ to $A_1(\widetilde{L(\Theta)})$. We have
$$f_{*}f^*(\Theta)=m\Theta$$
where $m$ is the degree of the isogeny $f$.
This gives us
$$f_*(\Theta_1+\cdots+\Theta_n)=(m-1)\Theta$$
as divisors on $P(\wt{C}/C)$. Also since $f_*(\Theta)=\Theta$, we have
$$f_*(\Theta_1+\cdots+\Theta_n-(m-1)\Theta)=0$$
and since the kernel of $f_*$ is torsion we have
$$k(\Theta_1+\cdots+\Theta_n)=k(m-1)\Theta\;,$$
modulo rational equivalence.

Therefore $f^*(k\Theta)=km\Theta$ for some positive integer $m$. Therefore $km\Theta$ is singular at a unique point on the Prym variety and hence $f^*(k\Theta)=km\Theta$ is smooth except one singular point on the Picard variety $\Pic^0(\widetilde{L(\Theta)})$. Note that the above equality is only true if $g$ is an injective map, because otherwise the image of $\Theta$ under $g$ might be having different singularities. This is a crucial point of difference with the case when there exists a codimension two cycle $Z$, such that $Z_*\circ AJ=N\id$. This is always the case for a cubic threefold as in this case we have a decomposition of diagonal of the cubic (whether integral decomposition exists or not is the question).  Therefore we have
$$km\Theta=f^*(k\Theta)$$
where $\Theta$ is realized as embedded in $A_1(\widetilde{L(\Theta)})\cong \Pic^0(\widetilde{L(\Theta)})$. Since $km\Theta$ is smooth except one singular point we have $f^*(k\Theta)$ is smooth except one singular point.

But now observe that $L(\Theta)$ can either be a curve or a surface. It can't be all of the cubic threefold because that $\Theta$ is a divisor in $P(\wt{C}/C)$. Because then any 1-cycle on $X$ is supported on $L(\Theta)$ as a cycle. So any point on the Prym variety  is coming from a point on $\Theta$. Now the composition $\Theta\to A_1(L(\Theta))\to P(\wt{C}/C)$ induced by the correspondence $L$ is a regular map of varieties, which is surjective, by the fact that $L(\Theta)=X$.  So a variety of dimension $g-2$ maps surjectively to a variety of dimension $g-1$. This cannot happen. It can't be a curve, because the singularity locus of the theta divisor of the Picard variety of a curve is of dimension greater than or equal to $g-4$.  Observe that the Jacobian of the curve is in isogeny with $P(\wt{C}/C)$ by the simplicity of the Prym variety (as the $\Theta$ divisor is embedded in the Jacobian, the Jacobian can't be isogenous to the trivial abelian variety). So it is a curve of genus equal to five, in such a case, the singularity locus of the theta divisor is positive dimensional. So the only possibility that $L(\Theta)$ is a surface. Then by Bertini's theorem consider a smooth hyperplane section of $\widetilde{L(\Theta)}$ which is a smooth projective curve of genus $g$. The Jacobian of the smooth projective curve is such that the kernel of the homomorphism from the Jacobian to the Albanese of the surface is simple [for reference please see \cite{Voi}(Proposition 1.4)].  Now observe that the Jacobian of the curve is isogenous to the two-fold product of the kernel and the Albanese variety of the desingularization  $\widetilde{L(\Theta)}$. If both the factors in the product is non-trivial, then the singularity locus of the theta divisor of the Jacobian must have an irreducible component of codimension two in the Jacobian. So the singularity locus of the theta divisor contains an irreducible component of dimension $g-2$. So the dimension of the singularity of the theta divisor of the Jacobian is $g-2$. Now the contribution from the singularity locus of the theta divisor of the Albanese is of codimension $5$ (because the singularity of the theta divisor of the Albanese is of zero dimensional). Therefore the codimension two component of the singularity of the theta divisor must come from singularity locus of the theta divisor of the kernel. Note that the kernel is a simple abelian variety. Therefore there cannot be a codimension two component of the singularity locus of the theta divisor of the kernel, this is by \cite{EL}. Hence the dimension of the singularity locus of the theta divisor of the curve is less than $g-3$, so it is either of dimension $g-3$ or $g-4$. In that case again using \cite{EL} the kernel must be trivial. So the Jacobian is isogenous to the Albanese of the desingularisation of $L(\Theta)$, which again is isogenous to the Prym variety. So the curve is of genus five and hence the singularity locus of the theta divisor of the Jacobian is positive dimensional but that of the Albanese is of zero dimensional. This is a contradiction. Therefore there does not exist an universal codimensional two cycle on the product $P(\wt{C}/C)\times X$.
\end{proof}

\begin{corollary}
\label{Cor1}
For a very general cubic threefold, the integral Chow theoretic decomposition does not hold. Consequently a very general cubic threefold is not stably rational.
\end{corollary}

\begin{proof}
If the integral Chow theoretic decomposition holds then we have the existence of a universal codimension two cycle on $P(\wt{C}/C)\times X$ by \cite{Voi1}[Theorem 1.6], which is false by the previous theorem. Since stable rationality implies integral Chow theoretic decomposition of the diagonal, we can conclude that a very general cubic threefold is not stably rational.
\end{proof}

\begin{corollary}
For a very general cubic threefold the cohomology class of $\Theta^{4}/4!$ is not algebraic on the intermediate Jacobian of $X$.
\end{corollary}
\begin{proof}
According to Theorem 1.7 of \cite{Voi1}, the universally triviality of Chow group of a smooth cubic threefold is equivalent to the algebraic-ness of the minimal class of $\Theta^{4}/4!$ is algebraic. As the universal triviality does not hold a very general smooth cubic threefold, the cohomology class of $\Theta^4/4!$ is not algebraic on  the intermediate Jacobian of $X$.
\end{proof}

\end{document}